\documentclass{amsart}

\usepackage[latin1]{inputenc}

\usepackage[T1]{fontenc}
\usepackage[english,frenchb]{babel}
\usepackage{amsmath,amssymb,amscd,latexsym,color,ifthen,enumerate}

\usepackage[dvips]{graphicx}

\usepackage[english]{babel}
\usepackage{verbatim}

\newcommand{\communique}{\leftrightarrow}

\newcommand{\N}{\mathbb{Z}_{+}}

\newcommand{\Zd}{\mathbb{Z}^d}
\newcommand{\C}{\mathbb{C}}

\newcommand{\R}{\mathbb{R}}
\newcommand{\Rd}{\mathbb{R}^d}

\renewcommand{\P}{\mathbb{P}}
\newcommand{\E}{\mathbb{E}}

\newcommand{\Ed}{\mathbb{E}^d}

\renewcommand{\epsilon}{\varepsilon}
\renewcommand{\phi}{\varphi}
\renewcommand{\limsup}{\overline{\lim}}

\newcommand{\miniop}[3]{%
\renewcommand{\arraystretch}{0.6}
\begin{array}{c}
{\scriptstyle #1}\\
#2\\
{\scriptstyle #3}
\end{array}
\renewcommand{\arraystretch}{1}}

\newcommand{\1}{1\hspace{-1.3mm}1}
\newcommand{\Var}{\text{Var }}

\newcommand{\A}[1]{A_{\ref{#1}}}
\newcommand{\alphaa}[1]{\alpha_{\ref{#1}}}
\newcommand{\B}[1]{B_{\ref{#1}}}
\newcommand{\betaa}[1]{\beta_{\ref{#1}}}
\renewcommand{\C}[1]{C_{\ref{#1}}}

\newcommand{\D}[1]{D_{\ref{#1}}}

\begin{document}

\newtheorem{theorem}{Théorème}[section]
\newtheorem{conjecture}[theorem]{Conjecture}
\newtheorem{lemme}[theorem]{Lemme}
\newtheorem{defi}[theorem]{Définition}
\newtheorem{coro}[theorem]{Corollaire}
\newtheorem{rem}[theorem]{Remarque}
\newtheorem{prop}[theorem]{Proposition}

%\color{blue}

\title[Déviations modérées de la distance chimique]{Déviations modérées de la distance chimique}
%\date{\today}

%\author{Olivier Garet}
%\author{R{é}gine Marchand}
%\address{Laboratoire de Math{é}matiques, Applications et Physique
%Math{é}matique d'Orl{é}ans UMR 6628\\ Universit{é} d'Orl{é}ans\\ B.P.
%6759\\
% 45067 Orl{é}ans Cedex 2 France}
%\email{Olivier.Garet@iecn.u-nancy.fr}
\author{Olivier Garet}
\address{Institut \'Elie Cartan Nancy (mathématiques)\\
Nancy-Université, CNRS \\
Boulevard des Aiguillettes B.P. 239\\
F-54506 Vand\oe uvre-lès-Nancy,  France.}
\email{Olivier.Garet@iecn.u-nancy.fr}
\author{Régine Marchand}
\email{Regine.Marchand@iecn.u-nancy.fr}

\def\motsclefs{Percolation, distance chimique, d\'eviations mod\'er\'ees, in\'egalit\'es de concentration, sous-additivit\'e.}


\begin{abstract}
Le but de cet article est l'étude des déviations modérées de la distance chimique, c'est à dire de la plus courte distance sur l'amas de percolation de Bernoulli.
Ainsi, on étudie la taille des fluctuations aléatoires autour de la valeur moyenne, ainsi que le comportement asymptotique de cette dernière. Ceci permet de préciser la convergence dans le théorème de forme asymptotique.
L'article se base principalement sur des estimées de concentration de Boucheron, Lugosi et Massart ainsi que sur la théorie des fluctuations des fonctions sous-additives d'Alexander.   
\end{abstract}
\thanks{Ce travail a bénéficié d'une aide de l'Agence Nationale de la Recherche
portant la référence ANR-08-BLAN-0190.}

\maketitle 

\section{Introduction et résultats}

On s'intéresse au modèle de percolation de Bernoulli surcritique sur les arêtes de $\Zd$, où $d\ge 2$ est un entier fixé.  
%On note $\|.\|_1$ la norme $L^1$ sur $\Zd$ et  $\|.\|_{\infty}$ la norme infinie. 
%Une arête est une paire $\{x,y\}$ d'éléments de $\Zd$ avec $\|x-y\|_1=1$. 
L'ensemble des arêtes de $\Zd$ est noté $\Ed$. Sur l'ensemble $\Omega=\{0,1\}^{\Ed}$, on considère la tribu produit et la probabilité $\P=\mathcal{B}(p)^{\otimes \Ed}$, sous laquelle les coordonnées $(\omega_e)_{e \in \Ed}$ sont des variables aléatoires indépendantes identiquement distribuées de loi de Bernoulli de paramètre $p$. Pour toute la suite, on fixe  un paramètre $p>p_c(\Zd)$, où $p_c(\Zd)$ est le seuil critique de percolation sur~$\Zd$.

On note $C_{\infty}$ l'amas infini de percolation. La distance chimique $D(x,y)$ entre deux point $x,y$ de $\Zd$ est la longueur, en nombre d'arêtes, du plus court chemin ouvert entre ces deux points. \`A grande distance, cette distance chimique est équivalente à une norme $\mu$ déterministe -- voir Garet et Marchand~\cite{GM-fpppc}.
En fait, on dispose même d'estimées de grandes déviations, comme nous l'avons montré dans~\cite{GM-large}: 
\begin{equation*}
\label{GDmu}
\forall \epsilon>0 \quad \miniop{}{\limsup}{\Vert y\Vert_1\to +\infty} \frac{\log \P \left(0\communique   y, \; \frac{D(0,y)}{\mu(y)}\notin (1-\epsilon, 1+\epsilon)\right)}{\Vert y\Vert_1}<0.
\end{equation*}
Notre but ici est de montrer les résultats de déviations modérées suivants: 
\begin{theorem}
\label{concentrationD}
Il existe une constante $\C{concenD}>0$ telle que
\begin{equation}
\forall y \in \Zd \quad \E \left( \left| D(0,y)-\mu(y)\right| \1_{\{0\communique   y\}} \right) \le \C{concenD}\sqrt{\|y\|_1}\log(1+\|y\|_1). \label{concenD}
\end{equation}
%\label{tfa}
Il existe des constantes $\A{equmoderter},\B{equmoderter},\C{equmoderter}>0$  telles que
pour tout $y$ dans $\Zd$ non nul, 
\begin{equation}
\forall x\in [\C{equmoderter} (1+\log \|y\|_1),\|y\|_1^{1/2}]\quad  \P \left( \frac{| D(0,y)-\mu(y)|}{\sqrt{\|y\|_1}}> x , \; 0 \communique y\right) \le \A{equmoderter} e^{- \B{equmoderter} x}. \label{equmoderter}
\end{equation}
Il existe une constante $\C{equtfa}>0$ telle que, $\P$-presque sûrement, sur l'événement $\{0 \communique \infty\}$, pour tout $t$ assez grand,
\begin{equation}
\mathcal{B}_{\mu}^0(t-\C{equtfa}\sqrt{t}\log t)\cap C_{\infty} \subset B^0(t)\subset  \mathcal{B}_{\mu}^0(t+\C{equtfa}\sqrt{t}\log t), \label{equtfa}
\end{equation}
où on note $B^0(t)=\{x \in\Zd: \;  D(0,x)\le t\}$ la boule de rayon $t$ pour la distance chimique et $\mathcal{B}_{\mu}^0(t)$ la boule de rayon $t$ pour la norme $\mu$.
\end{theorem}
Dans le travail de Kesten concernant la percolation de premier passage, l'analogue de l'estimée~(\ref{equmoderter}) s'étend pour $x \in[0, \|y\|_1]$. Cependant, pour $x\in[\sqrt{\|y\|_1},\|y\|_1]$, les inégalités de grandes déviations donnent de meilleurs résultats que les déviations modérées; par contre, nous ne pouvons descendre dans la zone $x\in[0,\C{equmoderter} (1+\log \|y\|_1)]$ pour des raisons techniques liées à un procédé d'approximation et de renormalisation.

Remarquons d'emblée que comme tous les points ne sont pas dans l'amas infini, la distance chimique entre deux points peut être infinie, ce qui pose des problèmes évidents. On en construit donc une variation de la façon suivante: pour $x\in \Zd$, on note $x^*$ le point de l'amas infini le plus proche de $x$ (pour la norme $\|.\|_1$). En cas d'égalité, on choisit $x^*$ de manière à minimiser $x^*-x$ pour une règle fixée à l'avance (par exemple l'ordre lexicographique). On définit alors
$$\forall x,y \in \Zd \quad D^*(x,y)=D(x^*,y^*).$$
Nous allons principalement travailler avec $D^*$, 
et il ne sera pas difficile, en contrôlant la différence entre ces deux quantités, de revenir à  la distance chimique $D$. 
Voici les résultats que nous obtenons pour la distance chimique modifiée $D^*$:
\begin{theorem}
\label{concentrationD*}
Il existe une constante $\C{equvar}>0$ telle que %\label{var}
\begin{equation}
\forall y \in \Zd \quad \Var D^*(0,y)\le \C{equvar}\|y\|_1\log(1+\|y\|_1). \label{equvar}
\end{equation}
% \label{moder}
Pour tout $\D{equmoder}>0$, il existe des constantes $\A{equmoder},\B{equmoder},\C{equmoder}>0$ telles que
pour tout $y$ dans $\Zd$ non nul, 
\begin{equation}
\forall x\in [\C{equmoder} (1+\log \|y\|_1),\D{equmoder}\sqrt{\|y\|_1}]\quad  \P \left( \frac{| D^*(0,y)-\E[D^*(0,y)] |}{\sqrt{\|y\|_1}}> x \right) \le \A{equmoder} e^{- \B{equmoder} x}. \label{equmoder}
\end{equation}
%\label{etlesperance}
Il existe une constante $\C{equetlesperance1}>0$  telle que
\begin{equation}
\forall y \in \Zd\backslash\{0\} \quad 0  \le  \E[D^*(0,y)] -\mu(y)  \le  \C{equetlesperance1} \sqrt{\|y\|_1}\log (1+\|y\|_1). \label{equetlesperance1}
\end{equation}
%\label{moderbis}
Les deux estimées précédentes donnent immédiatement l'existence de constantes $\A{equmoderbis},\B{equmoderbis},\C{equmoderbis}$ strictement positives telles que
pour tout $y$ dans $\Zd$ non nul, 
\begin{equation}
\forall x\in [\C{equmoderbis} (1+\log \|y\|_1),\sqrt{\|y\|_1}]\quad  \P \left( \frac{| D^*(0,y)-\mu(y)|}{\sqrt{\|y\|_1}}> x \right) \le \A{equmoderbis} e^{- \B{equmoderbis} x}. \label{equmoderbis}
\end{equation}
\end{theorem}
Comme en percolation de premier passage (voir l'article de Kesten~\cite{kesten-modere}), la preuve de ce résultat se décompose essentiellement en deux étapes:
\begin{itemize}
\item le contrôle des fluctuations de $D^*(0,x)$ autour de sa valeur moyenne (propriété de concentration),
\item le contrôle de l'écart entre la valeur moyenne de $D^*(0,x)$ et $\mu(x)$.
\end{itemize}

Le travail original de Kesten utilisait des techniques de martingales.
Des techniques analogues ont été déployées par Howard et Newman~\cite{MR1849171} pour la percolation de premier passage euclidienne~\cite{MR1452554} et par Pimentel~\cite{pimentel-preprint} pour le modèle de Vahidi-Asl et Wierner.
Depuis, les travaux de Talagrand~\cite{MR1361756} ont mis en lumière l'importance des techniques de concentration, et ont permis d'améliorer les estimées de Kesten. Ainsi, récemment, Benaïm et Rossignol~\cite{MR2451057} ont pu affiner les estimées sur la variance en percolation de premier passage.
Ici, nous allons contrôler les fluctuations~(\ref{equvar}) et~(\ref{equmoder})  à l'aide de récentes inégalités de Boucheron, Lugosi et Massart~\cite{MR1989444}, qui sont puissantes et plus simples à mettre en oeuvre que le résultat abstrait de Talagrand.

Le contrôle~(\ref{equetlesperance1}) de l'écart entre la valeur moyenne $\E[D^*(0,y)]$ et son équivalent $\mu(y)$ se base
habituellement sur les déviations modérées des fluctuations étudiées précédemment. Cela se voit très clairement dans les modèles à symétrie sphérique, comme chez Howard et Newman~\cite{MR1849171} -- on pourra aussi consulter le survol de Howard~\cite{MR2023652}. Un argument de symétrie peut également donner des preuves simples dans la direction principale, voir Alexander~\cite{MR1202516}.
Dans une direction quelconque, les choses sont beaucoup plus difficiles. 
Ici, nous avons choisi d'utiliser les techniques développées par Alexander~\cite{Alex97} dans son article sur l'approximation des fonctions sous-additives.

Le papier s'organise comme suit: dans une première partie, on donne une série d'estimées sur la distance chimique et sur la distance modifiée qui dérivent principalement des résultats d'Antal et Pisztora~\cite{AP96};
on montre aussi comment le théorème~\ref{concentrationD} contrôlant les déviations modérées pour la distance chimique découle du théorème~\ref{concentrationD*} relatif à $D^*$.
Dans une deuxième partie, on montre les estimées~(\ref{equvar}) et~(\ref{equmoder}) du théorème~\ref{concentrationD*}. On utilise ici les inégalités de concentration de Boucheron, Lugosi et Massart, ainsi qu'une technique de renormalisation à l'échelle mésoscopique.
La troisième partie est consacrée au contrôle de l'écart~(\ref{equetlesperance1}) entre $\mu(x)$ et $\E[D^*(0,x)]$, suivant la méthode d'Alexander.
%Comme Kesten~\cite{kesten-modere} et Alexander~\cite{Alex97}, on peut alors
%préciser le théorème de forme asymptotique.

%%%%%%%%%%%%%%%%%%%%%%%%%%%%%%%%%%%%%%%%

\section{Quelques estimées}

%%%%%%%%%%%%%%%%%%%%%%%%%%%%%%%%%%%%%%%%

\subsection{Classical estimates}

Pour  $x$ dans $\Zd$, on note $C(x)$ l'amas de percolation contenant $x$, $|C(x)|$ est son cardinal.

Grâce à Chayes, Chayes, Grimmett, Kesten et Schonmann~\cite{CCGKS}, on
peut contrôler le diamètre des clusters finis: il existe deux constantes strictement positives $\A{amasfini}$ et $\B{amasfini}$ telles que
\begin{equation}
\label{amasfini}
\forall r>0 \quad
 \P \left( |C(0)|<+\infty, \; C(0) \not\subset [-r,\dots,r]^d \right)
\le \A{amasfini}e^{-\B{amasfini} r}.
\end{equation}

On peut également controler la taille des trous dans l'amas infini: il existe deux constantes strictement positives  $\A{amasinfini}$ et $\B{amasinfini}$ 
telles que
\begin{equation}
\label{amasinfini}
\forall r>0 \quad
 \P \left( C_\infty \cap  [-r,\dots,r]^d=\varnothing \right)
\le \A{amasinfini}e^{-\B{amasinfini} r}.
\end{equation}
Quand $d=2$, ce résultat résulte des estimées de grandes déviations de Durrett
et Schonmann~\cite{DS}. Leurs méthodes peuvent être aisément transposées à  $d\ge 3$.
Néanmoins, quand $d\ge 3$, la manière la plus simple d'obtenir ce résultat est d'utiliser les résultats de Grimmett
and Marstrand~\cite{Grimmett-Marstrand} sur les slabs.

\subsection{Estimées pour la distance chimique}

Le lemme qui suit est une conséquence d'un résultat intermédiaire obtenu par Antal et Pisztora~\cite{AP96}. Ce lemme contient de fait les deux théorèmes obtenus par Antal et Pisztora dans leur article.
\begin{lemme}
Il existe des constantes $\alphaa{APexp},\betaa{APexp}$ strictement positives telles que
\begin{equation}
\label{APexp}
\forall y \in \Zd \quad \E[e^{\alphaa{APexp} \1_{\{0\communique y\}}D(0,y)}]\le e^{\betaa{APexp} \|y\|_1}.
\end{equation}
\end{lemme}
\begin{proof}
L'inégalité $(4.49)$ d'Antal et Pisztora~\cite{AP96} dit qu'il existe un entier $N$ et un réel $c>0$ tels que
\begin{equation*}
\label{apcache}
\forall \ell \ge 0\quad \P(0\communique y, D(0,y)>\ell)\le \P \left(\sum_{i=0}^n (|C_{i}|+1)>\ell c N^{-d}\right),
\end{equation*}
où les $C_{i}$ sont des ensembles aléatoires indépendants de même loi, tels
qu'il existe $h>0$ avec $\E[\exp(h |C_{i}|)]<+\infty$, et où $n$ est un entier
vérifiant $n\le \|y\|_1/2N$, ce qui entraîne que $n+1\le \|y\|_1$.
Ainsi, pour tout $\ell>0$,
\begin{eqnarray*}
\label{apcachezz}
\P(0\communique y, D(0,y)\ge \ell) & \le  & \P(0\communique y, D(0,y)>\ell /2)\\
& \le & \P \left( \sum_{i=0}^n (|C_{i}|+1)>\frac \ell 2c N^{-d} \right)\\
& \le & \P \left( \sum_{i=1}^{\|y\|_1} \frac{2N^d}c (|C_{i}|+1)\ge \ell \right).
\end{eqnarray*}
Bien sûr, la dernière inégalité reste vraie si $\ell\le 0$, ce qui prouve que
$$\1_{\{0\communique y\}}D(0,y)\preceq \mu^{*\|y\|_1},$$
où $\mu$ est la loi de 
$\frac{2N^d}c (|C_{0}|+1)$ et $\preceq$ d\'esigne la comparaison stochastique.
On peut ainsi prendre $\alphaa{APexp}=h\frac{c}{2N^d}$ et $\betaa{APexp}=\log \E [\exp(h |C_{0}|+1)]$.
\end{proof}

\begin{coro}
\label{danslecube}
Il existe des constantes $\rho,\A{gecart},\B{gecart},\A{cube},\B{cube}>0$ telles que 
\begin{eqnarray}
\forall y\in\Zd \quad  \forall t\ge\rho \|y\|_1\quad  \P(0\communique y, \; D(0,y)>t)& \le &  \A{gecart}\exp(-\B{gecart} t),\label{gecart} \\
\forall r>0 \quad \P
\left(
\begin{array}{c}
\exists x,y\in \{0,\dots,r\}^d: \\
 \1_{\{x\communique y\}}D(x,y)>\rho r
\end{array}
\right)& \le & \A{cube}\exp(-\B{cube} r).\label{cube}
\end{eqnarray}
\end{coro}

\begin{proof}
Soit $y \in \Zd$ et $t>0$. Avec l'inégalité de Markov et le lemme précédent,
$$
\P(0\communique y, \; D(0,y)>t)  \le e^{-\alphaa{APexp}t}\E[e^{\alphaa{APexp} \1_{ \{0\communique y \} }D(0,y)}] \le e^{\betaa{APexp}\|y\|_1-\alphaa{APexp}t}.
$$
Pour obtenir (\ref{gecart}), on prend alors par exemple $\rho=2\frac{\betaa{APexp}}{\alphaa{APexp}}$.

Soient $x$ et $y$ dans $\{0,\dots,r\}^d$, et $\rho$ comme précédemment. Avec l'inégalité de Markov et le lemme précédent,
\begin{eqnarray*}
& & \P(\exists x,y\in \{0,\dots,r\}^d: \; \1_{\{x\communique y\}}D(x,y)>\rho r) \\
& \le & \sum_{x,y \in \{0,\dots,r\}^d} \P( \1_{\{x\communique y\}}D(x,y)>\rho r ) \\
& \le & \sum_{x,y \in \{0,\dots,r\}^d} \exp(-\alphaa{APexp} \rho r) \E[\exp(\alphaa{APexp} \1_{\{x\communique y\}}D(x,y))] \\
& \le & \sum_{x,y \in \{0,\dots,r\}^d} \exp(-\alphaa{APexp} \rho r) \exp(\betaa{APexp} \|y-x\|_1) 
 \le  (1+r)^{2d} \exp(-\betaa{APexp} r),
\end{eqnarray*}
ce qui termine la preuve de (\ref{cube}).
%On pourra donc prendre $\rho=2\frac{\betaa{APexp}}{\alphaa{APexp}}$, $\B{cube}=\betaa{APexp}/2$ et $\A{cube}=\sup\{(1+r)^{2d} %e^{-\betaa{APexp}r/2 }, \;r\ge 0\}$.
Remarquons qu'on trouve également ce corollaire comme sous-produit de la preuve
d'Antal et Pisztora chez Dembo, Gandolfi et Kesten~\cite{DGK}, lemme 2.14.
\end{proof}

\subsection{Estimées pour la distance chimique modifiée $D^*$}

On montre ici des estimées pour $D^*$ comparables à celles obtenues pour la distance chimique $D$:
\begin{lemme}
\label{apstar}
Il existe des constantes $\rho^*, \A{apstare},\B{apstare},\alphaa{controlemomexp}, \betaa{controlemomexp},\A{controlnormedeux},\B{controlnormedeux}$ telles que l'on ait pour tout $y\in\Zd$:
%\regine{$\rho_2=\C{apstare}=\rho^*$}
\begin{eqnarray}
\forall t\ge \rho^*\|y\|_1\quad \P(D^*(0,y)>t)&\le& \A{apstare}\exp(-\B{apstare} t), \label{apstare}\\ \E[e^{\alphaa{controlemomexp}D^*(0,y)}]& \le & e^{\betaa{controlemomexp}\|y\|_1}, \label{controlemomexp} \\
%\E [\1_{\{D^*(0,y)>\rho_2\|y\|\}} \exp(\frac{\B{apstare}}2(D^*(0,y)-\rho_2\|y\|))]&\le & \A{apstare}\exp(-\B{apstare}\rho_2\|y\|)  \nonumber\\
\label{controlnormedeux}
\|(D^*(0,y)-\rho^*\|y\|_1)^+\|_2& \le &\A{controlnormedeux}\exp(-\B{controlnormedeux}\|y\|_1).
%\label{controlesperance}
%\E[D^*(0,y)]& \le & \rho^*\|y\|_1+\A{controlnormedeux}.
\end{eqnarray}
\end{lemme}

\begin{proof}
Posons $\lambda=\frac1{4\rho}$ et $\rho^*=2\rho$. Comme $D^*(0,y)=D(0^*,y^*)$, on a 
\begin{eqnarray*}
\P(D^*(0,y)>t)& \le & \P(\|0^*\|_1 \ge \lambda t)+ \P( \|y^*-y\|_1 \ge \lambda t)\\
& &+\miniop{}{\sum}{\substack{ \|a\|_1\le \lambda t \\ \|b-y\|_1\le \lambda t}}  \P(\1_{\{a\communique b \}}D(a,b)>t).
\end{eqnarray*}
Les deux premiers termes de la somme se contrôlent à l'aide de~(\ref{amasinfini}). Si $t\ge \rho^* \|y\|_1$, alors pour chaque terme de la somme on a: 
$$
\|a-b\|_1\le \|y\|_1+2\lambda t\le  \frac{t}{2\rho} +2\lambda t=\frac{t}{\rho},
$$ 
ce qui permet d'utiliser le contrôle donné par~(\ref{cube}), et termine la preuve de~(\ref{apstare}). 
%On peut ainsi appliquer le corollaire~\ref{tresgrandecart}, et donc
%$$\P(D^*(0,y)>t)\le 2\A{amasinfini}e^{-\B{amasinfini}\lambda t}+(1+\lambda t)^{2d} \A{gecart}\exp(-\B{gecart} t),$$
%ce qui donne le premier résultat. 
Maintenant, pour $y\in \Zd$ non nul, en utilisant~(\ref{apstare}), il vient
\begin{eqnarray*}
\E [e^{\alpha D^*(0,y)}] & = & 1+\int_0^{+\infty} \P(D^*(0,y)>t)\alpha e^{\alpha t}dt \\
& \le & e^{\alpha \rho^* \|y\|_1} +\int_{\rho^* \|y\|_1}^{+\infty} \A{apstare}\exp(-\B{apstare} t)\alpha e^{\alpha t}dt,
\end{eqnarray*}
ce qui donne~(\ref{controlemomexp}) en choisissant par exemple $\alphaa{controlemomexp}=\B{apstare}/2$.
% Maintenant,
% \begin{eqnarray*}& &\E [\1_{D^*(0,y)>\rho^*\|y\|} \exp(-\frac{\B{apstare}}2(D^*(0,y)-\rho_2\|y\|))]\\ & = & \int_0^{+\infty}\frac{\B{apstare}}2e^{\frac{\B{apstare}u}2} \P(D^*(0,y)>\rho_2\|y\|+u)\ du\\
% & \le & \A{apstare}\exp(-\B{apstare}\rho_2\|y\|)\int_0^{+\infty}\frac{\B{apstare}}2e^{-\frac{\B{apstare}u}2} \ du
% \end{eqnarray*}
De même, en utilisant encore~(\ref{apstare})
\begin{eqnarray*}
\E([(D^*(0,y)-\rho^*\|y\|_1)^+]^2)  & = &  \int_0^{+\infty} \P(D^*(0,y)-\rho^*\|y\|_1>t) 2t \ dt \\
& \le &  \A{apstare} \exp(-\B{apstare}\|y\|_1 )\int_0^{+\infty} 2t\exp(-\B{apstare} t) \ dt,
\end{eqnarray*}
ce qui donne~(\ref{controlnormedeux}).
\end{proof}

\subsection{Preuve du théorème~\ref{concentrationD}}

Nous allons voir maintenant comment le théorème~\ref{concentrationD} découle du théorème~\ref{concentrationD*}.
L'estimée~(\ref{amasfini}) contrôlant la taille des grands amas finis nous permet de montrer que sur l'événement $\{0 \communique y\}$, les quantités $D(0,y)$ et $D^*(0,y)$ co\"incident avec grande probabilité: en effet, sur l'événement $\{0\communique y, \; 0\communique \infty\}$, on a l'identité $D(0,y)=D^*(0,y)$, et on contrôle la probabilité
$$\P(0\communique y, \; 0 \not\communique \infty) \le \A{amasfini}e^{-\B{amasfini} \|y\|_1}.$$

\begin{proof}[Preuve du théorème~\ref{concentrationD}]
Commençons par la preuve de~(\ref{concenD}). Soit $y \in \Zd$.
\begin{eqnarray*}
& & \E (|D(0,y)-\mu(y)|\1_{\{0\communique y\}}) \\
 & = & 
\E (|D(0,y)-\mu(y)| \1_{\{0\communique y,0\communique\infty\}})
+ \E (|D(0,y)-\mu(y)|\1_{\{0\communique y\},0\not\communique\infty\}}).
\end{eqnarray*}
On a d'une part, avec l'inégalité de Cauchy-Schwarz et les estimées~(\ref{gecart}) et~(\ref{amasfini}),
\begin{eqnarray*}
&&  \E ( | D(0,y)-\mu(y) | \1_{\{0\communique y,0\not\communique\infty\}}) \\
&\le & \| (D(0,y)-\mu(y)) \1_{\{0\communique y\}} \|_2 \sqrt{\P(0\communique y,0\not\communique\infty)}\\
& \le &(\|D(0,y) \1_{\{0\communique y\}} \|_2+\mu(y)) \sqrt{\P(0\communique y,0\not\communique\infty)}\\
& =  & o \left( \sqrt{\|y\|_1}\log(1+\|y\|_1) \right).
\end{eqnarray*}
D'autre part, avec les estimées~(\ref{equetlesperance1}) et~(\ref{equvar}) du  théorème~\ref{concentrationD*},
\begin{eqnarray*} & &\E( | D(0,y)-\mu(y) | \1_{\{0\communique y,0\communique\infty\}})\\ & = &  \E( |D^*(0,y)-\mu(y)|\1_{\{0\communique y,0\communique\infty\}})\\
& \le & \E( | D^*(0,y)-\mu(y)| ) \\
& \le & |\E [ D^*(0,y)] -\mu(y)|+\sqrt{\Var(D^*(0,y))}\\
& \le & \C{equetlesperance1}\|y\|_1^{1/2}\log(1+\|y\|_1)+(\C{equvar}\|y\|_1\log(1+\|y\|_1))^{1/2},
\end{eqnarray*}
ce qui prouve~(\ref{concenD}).

Passons à la preuve de~(\ref{equmoderter}). Soit $y \in \Zd$. \\
Pour tout $x$ dans $ [\C{equmoderbis} (1+\log \|y\|_1),\sqrt{\|y\|_1}]$, l'estimée~(\ref{equmoderbis}) du  théorème~\ref{concentrationD*} et l'estimée~(\ref{amasfini}) assurent que 
\begin{eqnarray*}
 && \P \left( \frac{| D(0,y)-\mu(y)|}{\sqrt{\|y\|_1}}> x , \; 0 \communique y\right) \\
 & \le & \P \left( \frac{| D^*(0,y)-\mu(y)|}{\sqrt{\|y\|_1}}> x \right) +\P(0 \communique y, \; 0 \not\communique +\infty) \\
& \le & \A{equmoderbis} e^{- \B{equmoderbis} x} +  \A{amasfini}e^{-\B{amasfini} \|y\|_1},
\end{eqnarray*}
ce qui prouve~(\ref{equmoderter}), vu que $x \le \sqrt{\|y\|_1} \le \|y\|_1$.

La preuve de~(\ref{equtfa}) est standard à partir de~(\ref{equmoderter}) -- voir par exemple la preuve du théorème 3.1 dans Alexander~\cite{Alex97}.
\end{proof}

\section{Déviations modérées}

Nous allons ici montrer les résultats~(\ref{equvar}) et~(\ref{equmoder}) de concentration pour $D^*$ annoncés dans le théorème~\ref{concentrationD*}.
La preuve s'inspire évidemment de Kesten~\cite{kesten-modere}; cependant, pour pallier le défaut d'intégrabilité de la distance chimique, 
on met ici en oeuvre une procédure de renormalisation et d'approximation:
pour $t$ réel positif et $k$ dans $\Zd$, on note $\Lambda_k^t$ l'ensemble des
arêtes dont les centres sont plus proches du point $kt$ que de tout autre point du réseau
$t\Zd$ (en cas d'égalité, on utilise une règle déterministe arbitraire pour associer l'arête à une unique boîte). On dit qu'un point de $\Zd$ est dans $\Lambda_k^t$ s'il est extrémité d'une arête de $\Lambda_k^t$: ainsi, les $(\Lambda_k^t)_{k \in \Zd}$ forment une partition de $\Ed$, mais pas de $\Zd$.

Maintenant, on note $D^t(a,b)$ la distance obtenue à partir de la distance chimique $D$
comme suit: si deux points $x$ et $y$ sont à l'intérieur d'une même boîte
 $\Lambda_k^t$, on les relie par une arête supplémentaire, dite \emph{rouge}, de longueur
$Kt$, où $K$ est une constante absolue strictement supérieure à $4 \rho$. 

En prenant $t$ suffisamment grand, on s'assurera, grâce au lemme~\ref{pastropdiff} que 
$D^*(0,y)$ et $D^t(0^*,y^*)$ sont très proches.
%(en réalité on montrera même que la probabilité de l'événement $D^*(0,y)\ne D^t(0,y)$ est faible). Pour l'instant, on suppose juste que $t\le |y|$  ***regarder***.
Pour la quantité $D^t(0,y)$, on peut localiser les chemins optimaux dans une boîte déterministe, ce qui n'était pas le cas de $D^*$. On peut alors utiliser un théorème de concentration dû à Boucheron, Lugosi et Massart~\cite{MR1989444} pour contrôler les déviations de  $D^t(0,y)$ par rapport à $\E[D^t(0,y)]$:
\begin{lemme}
\label{devDt}
Pour tout $\C{devDgammax}>0$, il existe des constantes $\B{devDgammax},\gamma>0$ telles que pour tout $y$ dans $\Zd \backslash \{0\}$, \begin{equation}
\forall x \le  \C{devDgammax}\|y\|_1^{1/2} \quad  \P \left( \frac{|D^{\gamma x}(0,y)-\E D^{\gamma x}(0,y)|}{\sqrt{\|y\|_1}} \ge x \right)  \le  2\exp(- \B{devDgammax}x). \label{devDgammax}
\end{equation}
\end{lemme}

\begin{proof}
Soit $\C{devDgammax}>0$, $y \in \Zd$ et $t>0$ fixés. L'un des ingrédients de la preuve est l'existence de moments exponentiels pour $D^t(0,y)$, qui découle principalement de l'existence de moments exponentiels pour $D^*$. Comme on peut ne prendre que des arêtes rouges, on a 
\begin{equation}
\label{majDt}
D^t(0,y)\le Kt \left(\frac{\|y\|_1}t+1 \right)=K(\|y\|_1+t).
\end{equation}
Maintenant, montrons qu'il existe des constantes $\alpha, \beta, \eta$ strictement positives telles que pour tout $y \in \Zd \backslash \{0\}$ et tout $t>0$,
\begin{equation}
\label{momexpDt}
\log \E(\exp (\alpha D^t(0,y)) \le \beta \|y\|_1 +\eta t.
\end{equation}
Remarquons d'abord que l'inégalité triangulaire et l'inégalité~(\ref{majDt}) assurent que
\begin{eqnarray*}
D^t(0,y) 
& \le & D^t(0,0^*) + D^t(0^*,y^*) + D^t(y^*,y) \\
%& \le & K\|0-0^*\|_1 + Kt + D^*(0,y) + K\|y-y^*\|_1 +Kt \\
& \le & K( 2t + \|0-0^*\|_1 + \|y-y^*\|_1) + D^*(0,y).
\end{eqnarray*}
En utilisant l'inégalité de H\"older, on obtient
$$\E(\exp (\alpha D^t(0,y)) \le \exp(2\alpha K t)[\E \exp(3 \alpha K\|0-0^*\|_1)]^{2/3}[\E \exp(3 \alpha D^*(0,y))]^{1/3}.$$
Les inégalités~(\ref{amasinfini}) et~(\ref{controlemomexp}) permettent alors de conclure la preuve de~(\ref{momexpDt}).

D'après~(\ref{majDt}), tout chemin qui réalise $D^t(0,y)$ reste à l'intérieur d'une boîte finie déterministe, ce qui entraîne que $D^t(0,y)$ ne dépend que du contenu d'une famille finie de boîtes mésoscopiques que l'on numérotera de $1$ à $N$.
Soient $U_1,\dots, U_N$ les vecteurs aléatoires tels que $U_i$ contienne  les états des arêtes du graphe microscopique qui sont dans la boîte $i$.
Il existe une fonction $S=S_{y,t}$ telle que
$$D^t(0,y)=S(U_1,\dots,U_N).$$
Notons que les $(U_i)$ sont indépendants.
Soient $U'_1,\dots, U'_N$ des copies indépendantes des $U_1,\dots,U_N$;
posons $S^{(i)}=S(U_1,\dots,U_{i-1},U'_i,U_{i+1},\dots, U_N)$ et enfin
\begin{eqnarray*}
V_{+} & = & \E \left[ \sum_{i=1}^N (( S-S^{(i)})_{+})^2| U_1,\dots,U_N \right], \\
V_{-} & = & \E \left[ \sum_{i=1}^N (( S-S^{(i)})_{-})^2| U_1,\dots,U_N \right].
\end{eqnarray*}
On peut déjà noter, avec l'inégalité de Efron-Stein-Steele (voir Efron-Stein~\cite{MR615434} et Steele~\cite{MR840528} ou la proposition 1 dans Boucheron, Lugosi et Massart~\cite{MR1989444}) que
\begin{equation}
\Var D^t(0,y)\le\E[V_-].
%\le\frac{3^dK^2}2(t\E[D^t(0,y)]+t^2).
\label{ESS}
\end{equation}
De plus, le théorème~2 de Boucheron, Lugosi et Massart~\cite{MR1989444} donne les inégalités de concentration suivantes:
%\begin{prop}
pour tous $\lambda, \theta>0$ tels que $\lambda\theta<1$:
\begin{eqnarray}
\log\E[\exp(\lambda(S-\E[S]))] & \le & \frac{\lambda\theta}{1-\lambda\theta}\log\E \left[ \exp \left(\frac{\lambda V_+}{\theta}\right) \right], 
\label{BouLuMa+}\\
\log\E[\exp(-\lambda(S-\E[S]))] & \le & \frac{\lambda\theta}{1-\lambda\theta}\log\E \left[ \exp \left(\frac{\lambda V_-}{\theta}\right) \right].
\label{BouLuMa-}
\end{eqnarray}
%\end{prop}
On note $M^t(y,z)$ le plus court chemin pour $D^t$ entre $y$ et $z$, que l'on choisit, en cas d'ambigu\"ité, suivant une règle déterministe fixée.
Si l'on note $R_i$ l'événement: ``$M^t(0,y)$
passe par la boîte numéro $i$'', on  remarque que
$S^{(i)}-S\le Kt \1_{R_i}$.
Ainsi, $V_-$ est majoré par $K^2 t^2 Y$, où $Y$ est le nombre
de boîtes visitées par $M^t(0,y)$.
Comme $Y\le 3^d(1+D^t(0,y)/t)$, on obtient finalement
\begin{equation}
V_-\le 3^dK^2 t(D^t(0,y)+t) \label{majV-}.
\end{equation}
Cette inégalité et l'existence de moments exponentiels pour $D^t(0,y)$ vont nous permettre de contrôler les déviations dans un sens. Par contre, pour l'autre sens, nous ne pouvons majorer aussi simplement $V_+$, et nous utiliserons une variation de l'inégalité~(\ref{BouLuMa+}) qui nous a été communiquée par R. Rossignol et M. Théret.
\begin{lemme}
On suppose qu'il existe $\delta>0$, des fonctions r\'eelles $(\phi_i)_{1\le i\le n},(\psi_i)_{1\le i\le n}$, $(g_i)_{1\le i\le n}$  telles que pour tout $i$ on a
$$(S-S^{(i)})_-\le\psi_i(U'_i)\text{ et }(S-S^{(i)})_-^2\le\phi_i(U'_i)g_i(U_1,\dots,U_n)$$
et $\alpha_i=\E[e^{\delta\psi_i(U_i)}\phi_i(U_i)]<+\infty$.
Si on pose 
$$W=\sum_{i=1}^n \alpha_ig_i(U_1,\dots,U_n),$$
alors pour tout $\theta>0$ et tout $\lambda\in[0,\min(\delta,\frac1\theta))$, on a
\begin{eqnarray}
\label{boulimatata}
\log\E[\exp(\lambda(S-\E[S]))] & \le & \frac{\lambda\theta}{1-\lambda\theta}\log\E \left[ \exp \left(\frac{\lambda W}{\theta}\right) \right].
\end{eqnarray}
\end{lemme}
On a déjà remarqué que $S^{(i)}-S\le Kt\1_{R_i}$, ce qui donne les inégalités voulues avec $\phi_i=\psi_i=Kt$, $g_i=\1_{R_i}$ et $\alpha_i=Kt e^{\delta Kt}$; on obtient
$$W=K^2t^2 e^{\delta K t}Y \le 3^dK^2t e^{\delta K t}3^d(D^t(0,y)+t).$$
Afin de retrouver une inégalité analogue à~(\ref{majV-}), on choisit alors $\delta=1/t$:
\begin{equation}
\label{majW}
\max(V_-,W) \le 3^d K^2 e^{K} t(D^t(0,y)+t).
\end{equation}
Nous allons ainsi pouvoir traiter simultanément les deux termes, en prenant cependant en compte les deux conditions: $\lambda <1/\theta$ et $\lambda<1/t$. On prend les  constantes $\alpha,\beta,\eta$ donnée par~(\ref{momexpDt}), et on choisit
$\lambda,t>0$ tels que
$$\left\{ 
\begin{array}{l}
\lambda<1/t, \\
\lambda^2\le \frac12 \frac{\alpha}{ 3^dK^2 e^{K}t}.
\end{array}
\right.$$
En posant 
$\theta=\frac{\lambda 3^d K^2 e^{K} t}{\alpha}$, on a bien la condition $\theta\lambda\le 1/2$, ce qui permet d'utiliser les inégalités~(\ref{BouLuMa-}) et~(\ref{boulimatata}).
Comme $\frac{\lambda}{\theta}3^d K^2e^{K} t=\alpha$, on a, avec les estimées~(\ref{majW}) et~(\ref{momexpDt}),
\begin{eqnarray} 
\frac{\lambda\theta}{1-\lambda\theta}\log\E\left[\exp\left(\frac{\lambda V_-}{\theta}\right)\right] 
&\le & 2\lambda\theta\log\E[\exp(\alpha(D^t(0,y)+t))] \nonumber \\
& \le & \frac{2.3^d K^2 e^{K}}{\alpha} \lambda^2t(\beta\|y\|_1+(\eta+\alpha)t) \nonumber \\
& \le & L\lambda^2 (t\|y\|_1+t^2) \label{majorL}
\end{eqnarray}
pour toute constante $L \ge \frac{2.3^d K^2 e^{K}}{\alpha} \max(\beta,\eta+\alpha)$ et $L \ge \C{devDgammax}/2$ -- cette derni\`ere condition sera utilis\'ee a la fin de la preuve. Fixons une telle constante $L$. La même majoration donne aussi
\begin{equation}
\frac{\lambda\theta}{1-\lambda\theta}\log\E\left[\exp\left(\frac{\lambda W}{\theta}\right)\right] \le L\lambda^2 (t\|y\|_1+t^2).
\label{majorelle}
\end{equation}
Ainsi, pour tout $u>0$ et tous $\lambda,t>0$ tel que $\lambda <1/t$ et $\lambda^2\le \frac12 \frac{\alpha}{3^dK^2 e^K t}$, l'inégalité de Markov et les inégalités~(\ref{BouLuMa-})  et~(\ref{boulimatata}) assurent que
$$
\P \left( |D^t(0,y) -\E [D^t(0,y)]| >u \right)\le 2\exp(-\lambda u+L\lambda^2 (t\|y\|_1+t^2)).$$
Maintenant, en prenant 
$\displaystyle \lambda=\frac{x\|y\|_1^{1/2}}{2Lt(\|y\|_1+t)}<\frac{x}{2Lt\|y\|_1^{1/2}}$.
on~a
$$
\left\{ \begin{array}{l}
x \le 2L  \|y\|_1^{1/2}, \\
x^2 \le \frac{2\alpha}{3^d}\frac{L^2}{K^2 e^K}t\|y\|_1
\end{array} \right. 
\Rightarrow 
\left\{ \begin{array}{l}
\lambda<1/t, \\
\lambda^2\le \frac12 \frac{\alpha}{ 3^dK^2 e^{K}t}
\end{array} \right. 
$$
et il vient, en prenant $u=x\sqrt{\|y\|_1}$,
$$\P \left( \frac{|D^t(0,y)-\E[D^t(0,y)]|}{\sqrt{\|y\|_1}}\ge x \right) \le 2\exp \left( -\frac{x^2\|y\|_1}{4L t(\|y\|_1+t)}\right).$$
En prenant finalement $t=\gamma x$ avec $\gamma=\frac{3^d}{ \alpha} \frac{K^2 e^K}{L}$, on voit que
$$ x\le  2L\|y\|_1^{1/2} \Rightarrow \left\{ \begin{array}{l}
\lambda<1/t, \\
\lambda^2\le \frac12 \frac{\alpha}{ 3^dK^2 e^{K}t}
\end{array} \right. 
$$
et on obtient, pour  $x\le  \C{devDgammax}\|y\|_1^{1/2}\le  2L\|y\|_1^{1/2}$:
 $$\P \left( \frac{|D^{\gamma x}(0,y)-\E[ D^{\gamma x}(0,y) ] | }{\sqrt{\|y\|_1}}> x \right) \le 2\exp\left(-\frac{x}{4L\gamma  }\times \frac{1}{1+2L \gamma }\right),$$
ce qui termine la preuve du lemme~\ref{devDt}.
\end{proof}

\begin{lemme}
\label{pastropdiff}
Il existe des constantes
$\A{presquepareil},\B{presquepareil},\A{esperance},\B{esperance}$ strictement positives telles que pour tout $y\in\Zd$ et tout $t\le\|y\|_1$, on ait
\begin{eqnarray}
\P(D^t(0^*,y^*)\ne D^*(0,y)) & \le & \A{presquepareil} (1+\|y\|_1)^{2d} \exp(-\B{presquepareil} t), \label{presquepareil} \\
\|D^*(0,y) -  D^t(0^*,y^*)\|_2 & \le & \A{esperance}(1+\|y\|_1)^{d+1}\exp(-\B{esperance} t).\label{esperance}
\end{eqnarray}
\end{lemme}

\begin{proof}
On pose $\Gamma=\{x\in\Zd:\;\|x\|_1\le 3 \rho^*\|y\|_1)\}$.
Notons 
\begin{eqnarray*}
L & = & \{M^t(0^*,y^*)\subset \Gamma\}, \\
A  & = & \miniop{}{\cap}{a,b\in \Gamma:\; \|a-b\|_1\ge t}\{\1_{\{a\communique b\}}D(a,b)\le 2\rho\|a-b\|_1\}, \\
B & = & \miniop{}{\cap}{a\in \Gamma}\{a\in C_\infty\text{ ou }C(a)\subset [-t,\dots,t]^d\}.
\end{eqnarray*}
Disons que deux boîtes $\Lambda_k^t$ et $\Lambda_{\ell}^t$ sont $*$-adjacentes si $\|k-\ell\|_{\infty}=1$.
Pour $k\in\Zd$, on dit qu'une boîte $\Lambda_k^t$ est bonne si quel que soit le point $x$ dans la boîte $\Lambda_k^t$, quel que soit le point $y$ dans  la boîte $\Lambda_k^t$ ou dans  l'une des
$3^d-1$ boîtes $*$-adjacentes, si $x$ et $y$ communiquent, alors ils sont reliés par un chemin ouvert
dont la longueur ne dépasse pas $4\rho t$; notons encore
$$G=\miniop{}{\cap}{\|k\|_1\le 1+3\rho^*\|y\|_1/t}\{\Lambda^t_k\text{ est bonne.}\}.$$ 
Nous allons montrer que 
\begin{equation}
\label{labelleinclusion}
L \cap A\cap B\cap  G \subset \{D^t(0^*,y^*)=D^*(0,y)\}.
\end{equation}
Plaçons-nous sur l'événement $L \cap A\cap B\cap  G$ et considérons $M^t(0^*,y^*)$, le chemin minimal pour $D^t$ entre $0^*$ et $y^*$, inclus, grâce à l'événement $L$, dans la boîte $\Gamma$. 
Ce chemin est fait de trois sortes de portions: des suites d'arêtes rouges,
des suites d'arêtes du cluster infini et des suites d'arêtes de clusters finis. 
Pour passer d'un cluster fini au cluster infini, on doit nécessairement utiliser une arête rouge.
Soit $0^*=y_0,\dots, y_n=y^*$ la suite des points de la trajectoire $M^t(0^*,y^*)$ et soit 
$$i_0=\max \{i: \;D(0^*,y_i)=D^t(0^*,y_i)\}.$$
Pour montrer~(\ref{labelleinclusion}), il suffit de montrer que $i_0=n$.
Comme 
$$D(0^*,y_{i_0})=D^t(0^*,y_{i_0})\le D^t(0^*,y^*)\le D^*(0,y)<+\infty,$$
 le point $y_{i_0}$
est dans l'amas infini.

Supposons par l'absurde que $i_0<n$: l'arête entre $y_{i_0}$ et $y_{i_0+1}$ n'est pas une
arête ouverte du graphe $\Zd$, sinon cela contredirait la maximalité de $i_0$.
C'est donc une arête supplémentaire de longueur $Kt$ rajoutée entre deux
points d'une même boîte $\Lambda_k^t$, ou arête rouge: ainsi, $D^t(y_{i_0},y_{i_0+1})=Kt$.
Si $y_{i_0}$ et $y_{i_0+1}$ communiquent dans le graphe aléatoire, alors l'événement $G$ assure que
$D(y_{i_0},y_{i_0+1})\le 4 \rho t \le Kt$. On a alors $D(y_{i_0},y_{i_0+1})\le D^t(y_{i_0},y_{i_0+1})$, ce qui
contredit encore  la maximalité de $i_0$. Ainsi,
$y_{i_0}$ et $y_{i_0+1}$ ne communiquent pas, ce qui veut dire que
$y_{i_0+1}$ n'est pas dans le cluster infini.
Soit 
$$j_0=\inf\{j\in [i_0+1,\dots,n]: \; y_j\in C_{\infty}\}.$$ 
Remarquons que $j_0<+\infty$ car $y_n\in C_\infty$.
Le chemin entre $y_{i_0}$ et $y_{j_0}$ utilise alternativement des arêtes rouges et
des morceaux de clusters finis. Regardons ce chemin à l'échelle mésoscopique, c'est-à-dire que l'on considère le chemin composé des coordonnées des boîtes de taille $t$ visitées par le chemin.
On dit qu'un site utilisé par le chemin mésoscopique est rouge si la portion du chemin correspondant à la traversée de la boîte correspondante comprend
une arête rouge. Deux sites rouges consécutifs de ce chemin ne peuvent
être séparés de plus d'un site car il n'y a pas de cluster fini reliant
des boîtes non adjacentes (on est dans l'événement $B$). Ainsi plus de la moitié des sites du chemin 
mésoscopique sont rouges; remarquons aussi que l'arête joignant $y_{i_0}$ (dans l'amas infini) et $y_{i_0+1}$ (dans un amas fini) est nécessairement rouge. Il s'ensuit que
$$D(y_{i_0},y_{j_0})\ge D^t(y_{i_0},y_{j_0})\ge \frac{\|y_{i_0}-y_{j_0}\|_1}t \times\frac12\times  Kt=\frac{K}2\|y_{i_0}-y_{j_0}\|_1.$$
Si  $y_{i_0}$ et $y_{j_0}$ ne sont pas dans des boîtes $*$-adjacentes,
alors on a à la fois $\|y_{i_0}-y_{j_0}\|_1 \ge t$ et $D(y_{i_0},y_{j_0})\ge \frac{K}2\|y_{i_0}-y_{j_0}\|_1>2 \rho \|y_{i_0}-y_{j_0}\|_1$ , ce qui ne peut arriver puisque $A$ est réalisé.
Les points $y_{i_0}$ et $y_{j_0}$ sont donc dans des boîtes $*$-adjacentes. Mais alors l'événement $G$ assure que
$D(y_{i_0},y_{j_0})\le 4\rho t\le Kt\le D^t(y_{i_0},y_{j_0})$, ce qui contredit
encore la maximalité de $i_0$ et prouve l'inclusion~(\ref{labelleinclusion}).

Ainsi 
$\P(D^*(0,y)\ne D^t(0,y))\le \P(L^c)+\P(A^c)+\P(B^c)+\P(G^c)$. 

Nous allons maintenant contrôler les probabilités $\P(L^c)$, $\P(A^c)$, $\P(B^c)$ et $\P(G^c)$. 

Pour majorer $\P(L^c)$,  remarquons que comme $K\ge 1$, tout point de $M^t(0^*,y^*)$ est à une distance (en norme $\|.\|_1$) plus petite que
$D^t(0^*,y^*)$ de $0^*$. Comme $D^t(0^*,y^*)\le D^*(0,y)$, on a alors
\begin{eqnarray*}
\P(L^c)& = & \P(M^t(0^*,y^*)\not\subset  \Gamma)\\
& \le & \P \left( \|0-0^*\|_1 \ge  \rho^* \|y\|_1 \right)  + \P \left( D^*(0,y)\ge 2\rho^*\|y\|_1 \right)\\
 & \le & \A{amasinfini}\exp(-\B{amasinfini} \|y\|_1)+\A{apstare}\exp(-2\B{apstare} \rho^*\|y\|_1), 
\end{eqnarray*}
en utilisant les estimées (\ref{amasinfini}) et (\ref{apstare}).
Avec l'estimée~(\ref{gecart}),
\begin{eqnarray*}
\P(A^c) & = & \sum_{a,b\in \Gamma:\; \|a-b\|_1\ge t}\P( D(a,b)\ge 2\rho\|a-b\|_1) \\
& = & \sum_{a,b\in \Gamma:\; \|a-b\|_1\ge t} \A{gecart} \exp(-\B{gecart} 2\rho\|a-b\|_1) \\
& = & (1+2\rho_2\|y\|_1)^{2d} \exp(-2\B{gecart} \rho t).
\end{eqnarray*}
Avec l'estimée~(\ref{amasfini}),
$\P(B^c)  =  (1+3\rho^*\|y\|_1)^{d} \A{amasfini}\exp(-\B{amasfini} t)$.\\
Finalement, avec l'estimée~(\ref{cube}),
$\P(G^c)\le (1+3\rho^*\|y\|_1/t)^{d} \A{cube}\exp(-\B{cube} t)$. \\
On obtient alors l'estimée~(\ref{presquepareil}) en se rappelant $ t \le \|y\|_1$.
%, où l'on prend $\B{borne}=\B{apstare}\rho_2$ et $\A{borne}=\A{apstare}$.

Pour le dernier point, remarquons que
$$
0\le D^*(0,y)-D^t(0^*,y^*)\le\rho^*\|y\|_1\1_{\{D^*(0,y)\ne D^t(0^*,y^*)\}}+(D^*(0,y)-\rho^*\|y\|_1)^+.
$$
Ainsi,
\begin{eqnarray*}
& &\|D^*(0,y)-D^t(0^*,y^*)\|_2\\ & \le & \rho^*\|y\|_1\sqrt{\P(D^*(0,y)\ne D^t(0^*,y^*))}+\|(D^*(0,y)-\rho^*\|y\|_1)^+\|_2,
%& \le & \rho_2\|y\|_1\A{presquepareil}^{1/2} (1+|y|)^{d} \exp(-\frac{\B{presquepareil}}2 t)+\A{controlnormedeux}\exp(-\B{controlnormedeux}\|y\|_1) \\
%& \le &  \A{esperance} (1+\|y\|_1)^{2d+1}\exp(-\B{esperance} t)
\end{eqnarray*}
et on conclut avec~(\ref{presquepareil}) et~(\ref{controlnormedeux}), en utilisant encore une fois que $\|y\|_1\ge t$.
\end{proof}

Nous pouvons maintenant passer aux preuves des estimées~(\ref{equvar}) et~(\ref{equmoder}) du théorème~\ref{concentrationD*}. L'idée en est simple: on utilise les estimées obtenues pour $D^t$ dans le lemme~\ref{devDt}, tout en contrôlant, avec le lemme~\ref{pastropdiff}, l'erreur d'approximation entre $D^*$ et $D^t$.

%%%%%%%%%%%%%%%%%%%%%%%%%%%%%%%%%%%%%%%%%%%%%%%%
\begin{proof}[Preuve du théorème~\ref{concentrationD*}, estimée~(\ref{equvar})]
%%%%%%%%%%%%%%%%%%%%%%%%%%%%%%%%%%%%%%%%%%%%%%%%
On peut écrire
\begin{eqnarray*}
& &\Var D^*(0,y)\\ & \le & 2(\Var D^t(0,y)+\Var(D^*(0,y)-D^t(0,y))\\
& \le &2\Var D^t(0,y)+4 \left( \E(D^*(0,y)-D^t(0^*,y^*))^2 +\E(D^t(0^*,y^*) -D^t(0,y))^2\right).
\end{eqnarray*}
On prend $t=\frac{d+1}{\B{esperance}}\log(1+\|y\|_1)\le \|y\|_1$ pour $\|y\|_1$ assez grand.
\begin{itemize}
\item Comme~(\ref{controlnormedeux}) assure que
$\E[D^*(0,y)]=O(\|y\|_1)$, en utilisant les inégalités~(\ref{ESS}) et~(\ref{majV-}), il vient:
$$\Var D^t(0,y) \le 3^dK^2(t\E[D^*(0,y)]+t^2)=O(\|y\|_1 \log (1+\|y\|_1)).$$
\item L'inégalité~(\ref{esperance}) assure que 
\begin{eqnarray*}
\E(D^*(0,y)-D^t(0^*,y^*))^2 \le \A{esperance}^2 (1+\|y\|_1)^{2d+2}\exp(-2\B{esperance} t)=O(1).
\end{eqnarray*}
\item L'inégalité triangulaire pour $D^t$ puis l'inégalité~(\ref{majDt})assurent que 
\begin{eqnarray*}
|D^t(0^*,y^*) -D^t(0,y)| & \le & D^t(0,0^*)+D^t(y,y^*) \\
& \le & K(\|0^*\|_1 +\|y-y^*\|_1 +2t).
\end{eqnarray*}
L'inégalité de Minkowski et l'estimée~(\ref{amasinfini}) assurent alors que 
$$\|D^t(0^*,y^*) -D^t(0,y)\|_2 =O(\log (1+\|y\|_1)),$$
\end{itemize}
ce qui conclut la preuve de l'estimée~(\ref{equvar}) du théorème~\ref{concentrationD*}.
\end{proof}

%%%%%%%%%%%%%%%%%%%%%%%%%%%%%%%%%%%%%%%%%%%%%%%%%
\begin{proof}[Preuve du théorème~\ref{concentrationD*}, estimée~(\ref{equmoder})]
%%%%%%%%%%%%%%%%%%%%%%%%%%%%%%%%%%%%%%%%%%%%%%%%%
Remarquons tout d'abord qu'il suffit de prouver l'estimée pour tout $y$ suffisamment grand. 
Soit $\D{equmoder}>0$. Prenons le $\gamma$ donné par le lemme~\ref{devDt} avec $\C{devDgammax}=\D{equmoder}$. 
On pose 
$$\C{equmoder}=\max\left\{1,\frac {4d}{\gamma\B{presquepareil}},\frac{d+1}{\gamma \B{esperance}}\right\}.$$
Pour tout $y\in\Zd\backslash\{0\}$, si  $x\le \D{equmoder}\sqrt{\|y\|_1}$ et $x\ge \C{equmoder}(1+\log \|y\|_1)$,  alors, d'après~(\ref{esperance}) et~(\ref{majDt}),
\begin{eqnarray*}
&& |\E[D^*(0,y)]-\E [D^{\gamma x}(0,y) ]| \\
& \le & |\E [D^*(0,y)]-\E [D^{\gamma x}(0^*,y^*) ]| + |\E[ D^{\gamma x}(0^*,y^*)]-\E[ D^{\gamma x}(0,y) ] | \\
& \le &|\E [D^*(0,y)]-\E [D^{\gamma x}(0^*,y^*)] | +\E[ D^{\gamma x}(0,0^*)]+\E[ D^{\gamma x}(y,y^*)] \\
& \le & \A{esperance} (1+\|y\|_1)^{2d+1}\exp(-\B{esperance}\gamma \C{equmoder} (1+\log\|y\|_1))  \\
& & \quad \quad +K\E(\|0^*\|_1)+K\E(\|y-y^*\|_1)+2K\gamma x.
\end{eqnarray*}
Avec~(\ref{amasinfini}), on sait que $\E(\|0^*\|_1)=\E(\|y-y^*\|_1)<+\infty$. Donc, comme $\C{equmoder} \ge \frac{2d+1}{\gamma \B{esperance}}$, pour $y$ assez grand, si  $x\le\D{equmoder}\sqrt{ \|y\|_1}$ et $x\ge \C{equmoder}(1+\log \|y\|_1)$, on a
$$\frac{|\E [D^*(0,y)]-\E [D^{\gamma x}(0,y) ]|}{\sqrt{\|y\|_1}}  \le x/2,$$
ce qui entraîne que
\begin{eqnarray*}
&& \P \left( \frac{| D^*(0,y)-\E [D^*(0,y)]|}{\sqrt{\|y\|_1}}> x \right) \\
& \le & \P(D^*(0,y)\ne D^{\gamma x}(0^*,y^*))+\P \left(\frac{| D^{\gamma x}(0^*,y^*)-\E [D^{\gamma x}(0,y)]|}{\sqrt{\|y\|_1}}> x/2\right).
\end{eqnarray*}
Vu que $x\ge \C{equmoder}(1+\log \|y\|_1)$ et $\C{equmoder}\ge \frac {4d}{\B{presquepareil}\gamma}$, l'estimée~(\ref{presquepareil}) assure que
\begin{eqnarray*}
\P(D^*(0,y)\ne D^{\gamma x}(0^*,y^*)) & \le & \A{presquepareil} (1+\|y\|_1)^{2d} \exp(-\B{presquepareil} \gamma x) \\
& \le & \A{presquepareil}2^{2d} \exp \left( -\frac{\B{presquepareil}}2 \gamma x \right),
\end{eqnarray*}
Pour traiter le second terme, on écrit d'abord, avec l'estimée~(\ref{majDt}):
\begin{eqnarray*}
&& | D^{\gamma x}(0^*,y^*)-\E [D^{\gamma x}(0,y)]| \\
& \le & | D^{\gamma x}(0,y)-\E[ D^{\gamma x}(0,y)]| + D^{\gamma x}(0,0^*) +D^{\gamma x}(y,y^*) \\
& \le & | D^{\gamma x}(0,y)-\E [D^{\gamma x}(0,y)]| + K \|0^*\|_1+ K \|y-y^*\|_1+ 2K\gamma x
\end{eqnarray*}
Ainsi, pour tout $y$ assez grand,
\begin{eqnarray*}
&& \P \left(\frac{| D^{\gamma x}(0^*,y^*)-\E [D^{\gamma x}(0,y)] |}{\sqrt{\|y\|_1}}> x/2\right) \\
& \le & \P \left(\frac{| D^{\gamma x}(0,y)-\E [D^{\gamma x}(0,y)] |}{\sqrt{\|y\|_1}}> x/9\right)+ 2\P \left(\|0^*\|_1 \ge\frac{x\sqrt{\|y\|_1}}{9K} \right)
\end{eqnarray*}
Le lemme~\ref{devDt} permet de contrôler le premier terme, tandis que l'estimée~(\ref{amasinfini}) contrôle le second, 
ce qui donne le résultat voulu.
\end{proof}

%%%%%%%%%%%%%%%%%%%%%%%%%%%%%%%%%%%

\section{Comportement asymptotique de l'espérance}

%%%%%%%%%%%%%%%%%%%%%%%%%%%%%%%%%%%
Notre but ici est de démontrer l'estimée~(\ref{equetlesperance1}) du théorème~\ref{concentrationD*}.

La fonctionnelle $D^*$ hérite de la sous-additivité de $D$. On peut également remarquer que pour tout $a$ dans $\Zd$, la loi de $(D^*(x+a,y+a))_{x \in\Zd,y \in \Zd}$ ne dépend pas de $a$.
Ainsi, si on pose $h(x)=\E[D^*(0,x)]$, on a 
$$\forall x,y\in\Zd\quad h(x+y)\le h(x)+h(y).$$
La sous-additivité de $h$ permet de mettre en oeuvre les techniques développées par Alexander~\cite{Alex97} pour l'approximation des fonctions sous-additives. Commençons par démontrer la convergence la convergence de $h(ny)/n$ vers $\mu(y)$:

%%%%%%%%%%%%%%%%%%%%%%%%%
\begin{lemme}
\label{limite}
%$\displaystyle \forall y \in \Zd \quad \lim_{n \to +\infty} \frac{\E[D^*(0,ny)]}n=\mu(y).$
Pour tout $y\in\Zd$, la suite $\frac{D^*(0,ny)}n$ converge presque sûrement et dans $L^1$ vers $\mu(y)$.
En particulier $h(ny)/n$ converge vers $\mu(y)$.
\end{lemme}
%%%%%%%%%%%%%%%%%%%%%%%%%%%%%%%

\begin{proof}
Comme $D^*(0,y)$ est intégrable et $(D^*(x,y))_{x \in\Zd,y \in \Zd}$ est stationnaire, le théorème ergodique sous-additif nous dit
qu'il existe $\mu^*(y)$ tel que $D^*(ny)/n$ converge presque sûrement et dans $L^1$ vers $\mu^*(y)$. Il reste donc juste à voir que $\mu^*(y)$ coïncide avec $\mu(y)$. 
Or, d'après Garet et Marchand~\cite{GM-fpppc}, sur l'événement $\{0\communique\infty\}$, on peut presque sûrement trouver une suite de points $(n_ky)_{k\ge 1}$ avec $0\communique n_k y$ et $D(0,n_k y)/n_k$ converge vers $\mu(y)$ lorsque $k$ tend vers l'infini.
Bien sûr,  $D^*(0,n_k y)/n_k$ converge vers $\mu^*(y)$. Comme sur $\{0 \communique \infty\}$ on a l'égalité $D(0,n_k y)=D^*(0,n_k y)$, ceci entraîne $\mu(y)=\mu^*(y)$.
\end{proof}

Nous allons rappeler les résultats d'Alexander sur l'approximation des fonctions sous-additives.
Introduisons quelques notations dérivées de celles d'Alexander~\cite{Alex97}.
Pour  $M$ et $C$ des constantes positives, on note
$$
GAP(M,C)=\left\{
\begin{array}{c}
h:\Zd\to\R, \\
(\|x\|_1\ge M)\Rightarrow \left( \mu(x) \le h(x)\le \mu(x)+C\|x\|_1^{1/2}\log \|x\|_1 \right)
\end{array}
 \right\}.
$$
Pour $x \in \Rd$, on choisit une forme linéaire $\mu_x$  de $\Rd$ telle que $\mu_x(x)=\mu(x)$
et telle~que
$$\forall y\in \mathcal B_{\mu}^0(\mu(x))\quad \mu_x(y)\le\mu(x).$$
La grandeur $\mu_x(y)$ est la longueur pour $\mu$ de la projection de $y$ sur la droite passant par $0$ et $x$ suivant un hyperplan d'appui au convexe $\mathcal B_{\mu}^0(\mu(x))$ au point $x$.
Il est alors facile de voir que pour tout $y \in \Rd$, $|\mu_x(y)|\le \mu(y)$. On pose alors, pour $C$ constante positive, 
$$Q_x^h(C)=
\left\{
\begin{array}{c}
y\in\Zd:\|y\|_1\le (2d+1)\|x\|_1, \\
 \mu_x(y)\le\mu(x), \; h(y)\le\mu_x(y)+C\|x\|_1^{1/2}\log\|x\|_1
\end{array}
\right\}.$$
L'idée est que les éléments de $Q_x^h(C)$ permettent de réaliser un maillage de $\Zd$ avec des pas pour lesquels $\mu_x$ approche correctement $h$. 
On définit encore, pour $M>0,C>0,a>1$,
$$CHAP(M,C,a)=
\left\{
\begin{array}{c}
h:\Zd\to\R,\;(\|x\|_1\ge M)\Rightarrow
\left(
\begin{array}{c}
\exists \alpha \in[1,a], \\
x/\alpha\in \text{Co}( Q_x^h(C))
\end{array}
\right)
\end{array}
\right\},
$$
où $\text{Co}(A)$ désigne l'enveloppe convexe de $A$ dans $\Rd$. Les résultats d'Alexander sont les suivants:

\begin{lemme}[Alexander~\cite{Alex97}]
\label{lemalex}
Soit $h$ une fonction positive sous-additive sur $\Zd$ et $M>1,C>0,a>1$ des constantes fixées. On suppose que pour tout $x\in\Zd$ avec $\|x\|_1 \ge M$, il existe un entier $n$, un chemin $\gamma$ de $0$ à $nx$ et une suite de points $0=v_0,v_1,\dots,v_m=nx$ de $\gamma$ tels que $m\le an$ et dont les incréments $v_i-v_{i-1}$ sont tous dans $Q_x^h(C)$. Alors
$h\in CHAP(M,C,a)$
\end{lemme}

\begin{theorem}[Alexander~\cite{Alex97}]
\label{Alex2}
Soit $h$ une fonction positive sous-additive sur $\Zd$ et $M>1,C>0,a>1$ des constantes fixées. 
Si $h\in CHAP(M,C,a)$, alors $h\in GAP(M,C)$.
\end{theorem}

\begin{defi}
On appelle $Q_x^h(C)$-chemin toute suite $(v_0,\dots,v_m)$ telle que pour tout $i \in \{0,\dots,m-1\}$, $v_{i+1}-v_i \in Q_x^h(C)$.

Soit $\gamma=(\gamma(0),\dots, \gamma(n))$ un chemin simple du graphe $\Zd$. 
On considère l'unique suite d'indices $(u_i)_{0 \le i \le m}$ telle que
$$
\begin{array}{l}
u_0=0, \; u_m=n, \\
\forall i \in \{0, \dots, m-1\} \quad \forall j \in \{u_i+1, \dots, u_{i+1}\} \quad \gamma(j)-\gamma(u_i) \in Q_x^h(C), \\
\forall i \in \{0, \dots, m-1\} \quad \gamma(u_{i+1}+1)-\gamma(u_i) \not \in Q_x^h(C).
\end{array}
$$
Le $Q_x^h(C)$-squelette de $\gamma$ est alors la suite $(\gamma(u_i))_{0 \le i \le m}$.
\end{defi}

Pour la distance chimique modifiée $h(.)=\E [D^*(0,.)]$, nous allons montrer le résultat suivant:

\begin{prop}
\label{propositionalex} Il existe des constantes $M>1$ et $C>0$  telles que si $\|x \|_1 \ge M$, alors pour tout $n$ suffisamment grand, il existe un chemin de $0$ à $nx$ avec un $Q_x^h(C)$-squelette contenant moins de $2n+1$ sommets.
\end{prop}

Voyons d'abord comment cette proposition permet d'obtenir  l'inégalité~(\ref{equetlesperance1}) et d'achever la preuve du théorème~\ref{concentrationD*}.  

%%%%%%%%%%%%%%%%%%%%%%%%%%%%%%%%%
\begin{proof}[Preuve de l'inégalité~(\ref{equetlesperance1})
%, à partir de la proposition~\ref{propositionalex}
] 

La proposition~\ref{propositionalex} et le lemme~\ref{lemalex} assurent que $h(.)=\E [D^*(0,.)]$ est dans $CHAP(M,C,2)$, ce qui implique, via le théorème~\ref{Alex2}, que $h$ est dans $GAP(M,C)$. Ceci donne l'estimée~(\ref{equetlesperance1}) pour tout $y \in \Zd$ tel que $\|y\|_1 \ge M$, et donc pour tout $y \in \Zd$ quitte à augmenter la valeur de $\C{equetlesperance1}$.
\end{proof}
%%%%%%%%%%%%%%%%%%%%%%%%%%

Passons maintenant à la preuve de la proposition~\ref{propositionalex}.
On choisit désormais $h(.)=\E[D^*(0,.)]$, on prend $\beta$ et $C$ tels que
\begin{equation}
\label{choixbetaC}
0<\beta<\B{equmoder} \text{ et } C{}>\sqrt{2d}\left(\frac{d}{\beta}+\C{equmoder} \right), 
\end{equation}
et on pose $C'=48C{}$.
On définit
\begin{eqnarray*}
{Q}_x & = & {Q}^h_x(C'), \\
G_x & = & \{y \in \Zd: \; \mu_x(y)>\mu(x)\}, \\
\Delta_x & = & \{y \in {Q}_x: \; y \text{ adjacent à } \Zd \backslash {Q}_x, \; y \text{ non adjacent à } G_x\}, \\
D_x & = & \{y \in {Q}_x: \; y \text{ adjacent à } G_x\}.
\end{eqnarray*}

\begin{lemme}
\label{accroissements}
Il existe une constante $M$ telle que si $\|x \|_1\ge M$, alors
\begin{enumerate}
\item si $y \in Q_x$, alors $\mu(y) \le 2 \mu(x)$ et $\|y\|_1 \le 2d \|x\|_1$;
\item si $y \in \Delta_x$, alors $\E[D^*(0,y)]-\mu_x(y) \ge \frac{C'}2 \|x \|_1^{1/2}\log \|x\|_1$;
\item si $y \in D_x$, alors $\mu_x(y) \ge \frac56 \mu(x)$;
\item si $x$ est assez grand, alors $(\|y\|_1\le\|x\|_1^{1/2}) \Longrightarrow (y\in Q_x)$. 
\end{enumerate}
\end{lemme}

\begin{proof}
Les arguments sont simples et essentiellement déterministes. On pourra se reporter au lemme 3.3 dans Alexander~\cite{Alex97}, qui est l'analogue dans le cadre de la percolation de premier passage. On utilise en particulier le fait que $\E [D^*(0, \pm e_i)] \le \rho^*+A_{\ref{controlnormedeux}}<+\infty$, qui remplace l'intégrabilité du temps de passage d'une arête.
\end{proof}

On note $D(v_1,v_m;(v_i))$ la longueur d'un plus court chemin ouvert entre $v_1$ et $v_m$ contraint à passer dans cet ordre par chacun des $v_i$. Alors
\begin{lemme}
$$\displaystyle \lim_{\|x\|_1 \to +\infty}\P \left(
\begin{array}{c}
\exists m \ge 1 \quad \exists \text{ un } Q_x\text{-chemin } (v_0=0, \dots, v_m): \\ 
\displaystyle \sum_{i=1}^{m-1} \E [D^*(v_i,v_{i+1})] - D(v_1,v_m;(v_i)) >C{} m \|x\|_1^{1/2} \log\|x\|_1
\end{array} 
\right)=0.$$
\end{lemme}

La preuve de ce lemme repose sur un dénombrement des $Q_x$-chemins, sur le contrôle de moments exponentiels pour
$\E [D^*(v_i,v_{i+1})]-D(v_i,v_{i+1})$ et sur une inégalité de type~BK. Au regard du travail d'Alexander, il aurait été plus naturel de travailler avec des quantités du type $\E [D^*(v_i,v_{i+1})]-D^*(v_i,v_{i+1})$: malheureusement, la fonctionnelle $D^*$ n'est pas monotone, contrairement à $D$, ce qui compromet l'utilisation d'une inégalité de type~BK. On doit donc ici jongler une fois de plus entre $D$ et $D^*$.

\begin{proof}
Soit $m \ge1$ et $x$ assez grand fixés.
Soit $(v_0=0, \; v_1, \; \dots, \; v_m)$ un $Q_x$-chemin partant de $0$.
Le lemme précédent implique que pour tout $i$, $\|v_{i+1}-v_i\|_1 \le 2d \|x\|_1$. On pose
\begin{eqnarray*}
Y_i & = & \E [D^*(v_i,v_{i+1})]-D^*(v_i,v_{i+1}) \\
Z_i & = & \E [D^*(v_i,v_{i+1})]-D(v_i,v_{i+1}).
\end{eqnarray*}
Le résultat de déviations modérées nous permet dans un premier temps de contrôler certains moments exponentiels de $Y_i$ si $x$ est assez grand. En effet, on écrit
$$\E \left[ \exp \left( \frac{\beta(Y_i)_+}{\sqrt{2d\|x\|_1}} \right) \right] =1+\int_{0}^{+\infty} \beta e^{\beta t} \P \left( {Y_i}\ge t \sqrt{2d\|x\|_1} \right)\ dt.$$
Remarquons que l'estimée~(\ref{controlnormedeux}) nous donne déjà la majoration simple:
$$
\max(Y_i,Z_i)\le \E [D^*(v_i,v_{i+1})]\le \A{controlnormedeux}+\rho^* \|v_{i+1}-v_i\|_1;
$$
en remarquant que $\|v_{i+1}-v_i\|_1 \ge 1$, ceci implique en particulier, dès que $\|x\|_1$ est assez grand, que
\begin{equation}
\label{chemisederegine}
\frac{\max(Y_i,Z_i)}{\sqrt{2d\|x\|_1}} %\le \frac{\A{controlnormedeux}+\rho^* \|v_{i+1}-v_i\|_1}{\sqrt{2d\|x\|_1}} 
\le 2 \rho^* \sqrt{\|v_{i+1}-v_i\|_1}.
\end{equation}
Ceci assure que si $t>2 \rho^* \sqrt{\|v_{i+1}-v_i\|_1}$, alors 
$\P( {Y_i}\ge t\sqrt{2d\|x\|_1})=0$.
D'autre part, le résultat de déviations modérées~(\ref{equmoder}) nous dit qu'il existe des constantes $\A{equmoder},\B{equmoder}>0$ telles que si $\C{equmoder}(1 +\log \|v_{i+1}-v_i\|_1) \le t \le 2 \rho^* \sqrt{\|v_{i+1}-v_i\|_1}$, alors, comme $\|v_{i+1}-v_i\|_1 \le 2d \|x\|_1$,
$$\P({Y_i}\ge t\sqrt{2d\|x\|_1}) \le \P({Y_i}\ge t\|v_{i+1}-v_i\|_1^{1/2})\le \A{equmoder} \exp(-\B{equmoder} t).$$
Ainsi, comme $\beta<\B{equmoder}$, on a
\begin{eqnarray}
 \E \left[ \exp \left(  \frac{\beta(Y_i)_+}{\sqrt{2d\|x\|_1}} \right) \right]
& \le & 1+\int_0^{\C{equmoder}(1 +\log (2d\|x\|_1))}\beta e^{\beta t} dt + \int_0^{+\infty} \beta e^{\beta t} \A{equmoder} e^{-\B{equmoder} t} dt \nonumber \\
%& \le & \exp(\beta a_1(1 +\log \|x\|)) +\frac{a_2 \beta}{a_3-\beta} \\
& \le & 1+ (2de)^{\beta \C{equmoder}} \|x\|_1^{\beta \C{equmoder}} +\frac{\A{equmoder} \beta}{\B{equmoder}-\beta}. %\le (6d\|x\|_1)^{\C{equmoder}\beta}  
\label{momoexp}
\end{eqnarray}

Remarquons ici que notre contrôle des moments exponentiels de $Y_i$ est moins bon que dans le cas de la percolation de premier passage, où la borne, analogue à~(\ref{momoexp}), obtenue par Alexander est indépendante de $\|x\|_1$: ceci est dû à la renormalisation que nous avons utilisée pour obtenir les déviations modérées~(\ref{equmoder}).

Notons que
\begin{itemize}
\item si $v_i \not \communique v_{i+1}$, alors $(Z_i)_+=0$;
\item si $v_i \communique  v_{i+1}$ et $v_i \communique \infty$, alors $(Z_i)_+=(Y_i)_+$;
\end{itemize}
Ainsi, en utilisant la majoration~(\ref{chemisederegine}),
%\begin{eqnarray*}
$$
\exp \left(  \frac{\beta(Z_i)_+}{\sqrt{2d\|x\|_1}} \right) 
 \le  1+\exp \left( \frac{\beta (Y_i)_+}{\sqrt{2d\|x\|_1}} \right)
 +
 \1_{ \left\{ \substack{v_i\communique v_{i+1} \\ v_i\not\communique\infty} \right\}}
 \exp \left( 2\beta\rho^* \sqrt{\|v_{i+1}-v_i\|_1} \right).
 $$
%\end{eqnarray*}
En utilisant~(\ref{amasfini}) puis~(\ref{momoexp}), on obtient alors, pour tout $x$ assez grand:
\begin{eqnarray}
& & \E \left[ \exp \left(  \frac{\beta(Z_i)_+}{\sqrt{2d\|x\|_1}} \right) \right] \label{mimiexp} \\
& \le & 1+ \E \left[ \exp \left(  \frac{\beta(Y_i)_+}{\sqrt{2d\|x\|_1}} \right) \right] +\exp \left( 2\beta\rho^* \sqrt{\|v_{i+1}-v_i\|_1} \right) \P \left(
\begin{array}{c}
v_i \communique v_{i+1}\\
v_i\not\communique \infty
\end{array}
\right)\nonumber  \\
& \le & 1+ \E \left[ \exp \left( \frac{\beta (Y_i)_+}{\sqrt{2d\|x\|_1}} \right) \right] +\A{amasfini} \exp \left( 2\beta\rho^* \sqrt{\|v_{i+1}-v_i\|_1}-\B{amasfini} \|v_{i+1}-v_i\|_1 \right)\nonumber\\
& \le &  (6d\|x\|_1)^{\beta\C{equmoder}}. \nonumber
\end{eqnarray}
On peut appliquer une inégalité de type BK à la quantité $D(v_1,v_m;(v_i))$: avec  le théorème 2.3 de Alexander~\cite{MR1202516}, pour tout $t>0$, si les $Z'_i$ sont des copies indépendantes des $Z_i$, alors, avec~(\ref{mimiexp}),
\begin{eqnarray*}
&& \P \left( \sum_{i=1}^{m-1} \E [D^*(v_i,v_{i+1})] - D(v_1,v_m;(v_i))>C{} m \|x\|_1^{1/2} \log\|x\| _1\right) \\
& \le & \P \left(\sum_{i=1}^{m-1} Z_i'>C{} m \|x\|_1^{1/2} \log\|x\|_1 \right) \\
%& \le & \P \left( \exp \left( \frac{\beta}{\sqrt{2d}}\sum_{i=1}^{m-1} Y_i'\right)>\exp \left( \frac{\beta}{\sqrt{2d}} C{} m \|x\|^{1/2} \log\|x\| \right)\right)\\
& \le & \exp \left(-\frac{\beta C{} m \log\|x\|_1}{\sqrt{2d}} \right)\prod_{i=1}^{m-1}\E\left[\exp \left( \frac{ \beta Z_i}{\sqrt{2d\|x\|_1}} \right)\right]\\
& \le & \left( (6d)^{\beta\C{equmoder}} \|x\|_1^{-\frac{\beta C{}}{\sqrt{2d}} + \beta\C{equmoder}} \right) ^m.
\end{eqnarray*} 
Le lemme~\ref{accroissements} donne l'existence d'une constante $K$ telle qu'il y a au plus $(K\|x\|^d)^m$ $Q_x$-chemins de $m+1$ sommets issus de $0$, donc, en sommant sur tous 
ces chemins, on obtient
\begin{eqnarray*}
 &  & \P \left(
\begin{array}{c}
\exists \text{ un } Q_x\text{-chemin } (v_0=0, \dots, v_m): \\ 
\displaystyle \sum_{i=1}^{m-1}  \E [D^*(v_i,v_{i+1})] - D(v_1,v_m;(v_i)) >C{} m \|x\|_1^{1/2} \log\|x\|_1
\end{array} 
\right) \\
& \le & \left( K(6d)^{\beta\C{equmoder}} \|x\|_1^{d-\frac{\beta C{}}{\sqrt{2d}} + \beta\C{equmoder}} \right) ^m = \left( \frac{L}{\|x\|_1^{\alpha}} \right)^m,
\end{eqnarray*}
pour deux constantes $L$ et $\alpha$; le choix~(\ref{choixbetaC}) que nous avons fait pour $\beta$ et $C{}$ assure de plus $\alpha>0$. Ainsi,
si $x$ assez grand, $L \|x\|_1^{-\alpha}\le \frac12 $, de sorte qu'en sommant sur les longueurs $m$ possibles:
$$
\P \left(
\begin{array}{c}
\exists m \ge 1 \quad \exists \text{ un } Q_x \text{-chemin } (v_0, \dots, v_{m}): \\    
\displaystyle \sum_{i=1}^{m-1}  \E [D^*(v_i,v_{i+1})] - D(v_1,v_m;(v_i))  >C{} m \|x\|_1^{1/2} \log\|x\|_1
\end{array} 
\right) \\
 \le  \frac{2 L}{ \|x\|_1^{\alpha}},
$$
ce qui termine la preuve du lemme.
\end{proof}
%%%%%%%%%%%%%%%%%%%%%%%%%%%%%%%%%%%%%%%%%%%%

\begin{proof}[Preuve de la proposition~\ref{propositionalex}]
Cette proposition est l'analogue de la proposition 3.4. de Alexander~\cite{Alex97}. 

La preuve de ce résultat déterministe utilise la "méthode probabiliste": nous allons montrer qu'avec une probabilité strictement positive, on peut contruire un tel $Q_x$-chemin issu de $0$ à partir d'un chemin réalisant $D^*(0,nx)$.
Avec le lemme précédent et l'estimée~(\ref{amasinfini}), on peut trouver $M>1$ tel que si $\|x\|_1\ge M$, 
\begin{eqnarray*}
 \P\left(
\begin{array}{c}
\exists m \ge 1 \quad \exists \text{ un } Q_x\text{-chemin } (v_0=0, \dots, v_m): \\ 
\displaystyle \sum_{i=1}^{m-1}  \E [D^*(v_i,v_{i+1}) ]- D(v_1,v_m;(v_i)) >C{} m \|x\|_1^{1/2} \log\|x\|_1
\end{array} 
\right) & \le & \frac15, \\
\text{et } \P \left(\|0^*\| \ge \|x\|_1^{1/2} \right) & \le & \frac15.
\end{eqnarray*}
Quitte à augmenter $M$, on peut de plus supposer que si $\|y\|_1\le \|x\|_1^{1/2}$ et si $\|x\|_1\ge M$, alors $y \in Q_x$.
On fixe alors $x\in \Zd$ avec $\|x\|_1\ge M$. %L'estimée~(\ref{gddstar}) 
Le lemme~\ref{limite} assure qu'il existe un $n_0 \in \N$ tel que
$$\forall n \ge n_0 \quad \P(D^*(0,nx)> n(\mu(x)+1))\le \frac15.$$
On fixe alors un $n \ge n_0$. Avec probabilité au moins $1/5$, les quatre propriétés suivantes sont donc satisfaites:
\begin{enumerate}[a)]
\item pour tout $m \ge 1$, pour tout $Q_x$-chemin $(v_0=0, \dots, v_m)$,
$$\sum_{i=1}^{m-1} \E [D^*(v_i,v_{i+1})] - D(v_1,v_{m};(v_i))\le C{} m \|x\|_1^{1/2} \log\|x\|_1,$$
\item $D^*(0,nx) \le n(\mu(x)+1)$,
\item $\|0^*\| _1\le \|x\|_1^{1/2}$,
\item $\|(nx)^*-nx\|_1 \le \|x\|_1^{1/2}$.
\end{enumerate}
On peut donc trouver un $\omega$ satisfaisant ces quatre propriétés, et nous allons travailler pour la suite avec cette réalisation $\omega$ particulière. Soit $v_1=0^*,\dots, v_{m}=(nx)^*$ le $Q_x$-squelette d'un chemin $\gamma$ réalisant
la distance chimique de $0^*$ à $(nx)^*$. On pose $v_0=0$, $v_{m+1}=nx$ : les propriétés c) et d) assurent qu'on obtient ainsi un $Q_x$-chemin de $0$ à $nx$. Nous allons montrer que $m+2 \le 2n+1$, ce qui terminera la preuve de la proposition. Remarquons que, par construction, 
 $$D(v_1,v_{m};(v_i))=D^*(0,nx).$$
Notre $Q_x$-chemin satisfait, pour ce $\omega$ particulier, d'après a),
\begin{equation}
\label{numero}
  \sum_{i=1}^{m-1} \E[D^*(v_i,v_{i+1})] - D(v_1,v_{m};(v_i))\le C{} m \| x\|_1^{1/2} \log\|  x\|_1.
\end{equation}
% ce qui implique, avec~(\ref{controlnormedeux}), et quitte à augmenter $M$,
% \begin{eqnarray*}
%  \sum_{i=0}^{m-1} \E D^*(v_i,v_{i+1}) - D^*(0,nx) & \le & \rho^*\|v_1\|+ A_{\ref{controlnormedeux}}+    C{} m \| x\|_1^{1/2} \log\| x\|_1\\
%  & \le &\rho^* \|x\|_1^{1/2}+ A_{\ref{controlnormedeux}}+    C{} m \| x\|_1^{1/2} \log\| x\|_1\le 2C{} m \| x\|_1^{1/2} \log\| x\|_1.
% \end{eqnarray*}
D'autre part, par définition de $Q_x$, avec la propriété d), le fait que $\mu$ est une norme,
et l'estimée~(\ref{gecart}),
\begin{eqnarray*}
m\mu(x)&\ge& \sum_{i=0}^{m-1}\mu_x(v_{i+1}-v_i)=\mu_x((nx)^*)\\
& =& \mu_x((nx))+\mu_x((nx)^*-nx)\\
& \ge& n\mu(x)-\mu((nx)^*-nx)
 \ge   n\mu(x)-\rho\|x\|_1^{1/2}.
\end{eqnarray*}
 Ainsi, quitte à augmenter $M$, on obtient que  $n\le \frac{11}{10}m$. Maintenant, avec b) et quitte à augmenter encore $M$ si nécessaire,
\begin{eqnarray*} 
\sum_{i=1}^{m-1} \E [D^*(v_i,v_{i+1})] 
& \le & D^*(0,nx)+ C{} m \| x\|_1^{1/2} \log\| x\|_1\\
&\le  &   n(\mu(x)+1)+C{} m \| x\|_1^{1/2} \log\| x\|_1\\
& \le & n\mu(x)+2C{} m \| x\|_1^{1/2} \log\| x\|_1.
\end{eqnarray*}
Nous allons maintenant distinguer, dans le $Q_x$-squelette, les accroissements courts et les accroissements longs:
\begin{eqnarray*}
S((v_i)) & = & \{i: \; 1 \le i \le m-1, \; v_{i+1}-v_i \in \Delta_x\}, \\
L((v_i)) & = & \{i: \; 1 \le i \le m-1, \; v_{i+1}-v_i \in D_x\}.
\end{eqnarray*}
Remarquons que la définition du $Q_x$-squelette implique que ces deux ensembles forment bien une partition de $\{1,\dots,m-1\}$. 
Majorons tout d'abord le nombre d'accroissements courts à l'aide du lemme~\ref{accroissements}, estimée~(2): rappelons $\mu_x(y)\le \mu(y)\le \E[D^*(0,y)]$, et donc
\begin{eqnarray*}
\sum_{i=1}^{m-1}\E [D^*(v_i,v_{i+1})] & = & \sum_{i=1}^{m-1}[ \mu_x(v_{i+1}-v_i)+(\E D^*(v_i,v_{i+1})-\mu_x(v_{i+1}-v_i))]\\
& \ge & \mu_x((nx)^*)-\mu_x(0^*)+|S((v_i))| \frac{C'}2 \|x\|_1^{1/2}\log\|x\|_1\\
& \ge & n\mu(x)-2\rho\|x\|_1^{1/2}+|S((v_i))| \frac{C'}2 \|x\|_1^{1/2}\log\|x\|_1.
\end{eqnarray*}
Ainsi, en cumulant les deux dernières estimées, quitte à augmenter $M$,
\begin{eqnarray*}
|S((v_i))| \frac{C'}2 \|x\|_1^{1/2}\log\|x\|_1 & \le & 2\rho\|x\|_1^{1/2} +2C{} m \| x\|_1^{1/2} \log\| x\|_1 \\
& \le & 3C{} m \| x\|_1^{1/2} \log\| x\|_1,
\end{eqnarray*}
d'où
$$|S((v_i))|\le 6\frac{C{}}{C'}m=\frac{m}{8}.$$
De la même manière, majorons le nombre d'accroissements longs à l'aide du lemme~\ref{accroissements}, estimée~(3): 
\begin{eqnarray*}
\sum_{i=1}^{m-1}\E[ D^*(v_i,v_{i+1})]&=&\sum_{i=1}^{m-1}[ \mu_x(v_{i+1}-v_i)+(\E [D^*(v_i,v_{i+1})]-\mu_x(v_{i+1}-v_i))]\\
& \ge & n\mu(x)-2\rho\|x\|_1^{1/2}+\frac56 |L((v_i))|\mu(x) \ge \frac56 |L((v_i))|\mu(x),
\end{eqnarray*}
quitte à augmenter $M$; ceci donne
$$|L((v_i))|\le\frac{6}{5}n+2C{} m \frac{\| x\|_1^{1/2} \log\| x\|_1}{\mu(x)}\le\frac{6}{5}n+m/8.$$
Finalement,
$m=|S((v_i))|+|L((v_i))|+\le \frac{6}{5}n+m/4$,
d'où $m\le\frac{8}{5}n$, ce qui termine la preuve.
\end{proof}

Nous tenons à remercier Raphaël Rossignol et Marie Théret qui nous ont gentiment communiqué l'extension du résultat de Boucheron, Lugosi et Massart.

%%%%%%%%%%%%%%%%%%%%%%%%%%%%%%%%%%%%%%%%%%%%%%%%%%

%%%%%%%%%%%%%%%%%%%%%%%%%%%%%%%%%%%%%%%%%%%%%%%%%%%%%%%%%%%%

\def\refname{Bibliographie}
\bibliographystyle{plain}
\bibliography{mdev-fr}

\begin{thebibliography}{10}

\bibitem{MR1202516}
Kenneth~S. Alexander.
\newblock A note on some rates of convergence in first-passage percolation.
\newblock {\em Ann. Appl. Probab.}, 3(1):81--90, 1993.

\bibitem{Alex97}
Kenneth~S. Alexander.
\newblock Approximation of subadditive functions and convergence rates in
  limiting-shape results.
\newblock {\em Ann. Probab.}, 25(1):30--55, 1997.

\bibitem{AP96}
Peter Antal and Agoston Pisztora.
\newblock On the chemical distance for supercritical {B}ernoulli percolation.
\newblock {\em Ann. Probab.}, 24(2):1036--1048, 1996.

\bibitem{MR2451057}
Michel Bena{\"{\i}}m and Rapha{\"e}l Rossignol.
\newblock Exponential concentration for first passage percolation through
  modified {P}oincar\'e inequalities.
\newblock {\em Ann. Inst. Henri Poincar\'e Probab. Stat.}, 44(3):544--573,
  2008.

\bibitem{MR1989444}
St{\'e}phane Boucheron, G{\'a}bor Lugosi, and Pascal Massart.
\newblock Concentration inequalities using the entropy method.
\newblock {\em Ann. Probab.}, 31(3):1583--1614, 2003.

\bibitem{CCGKS}
J.~T. Chayes, L.~Chayes, G.~R. Grimmett, H.~Kesten, and R.~H. Schonmann.
\newblock The correlation length for the high-density phase of {B}ernoulli
  percolation.
\newblock {\em Ann. Probab.}, 17(4):1277--1302, 1989.

\bibitem{DGK}
Amir Dembo, Alberto Gandolfi, and Harry Kesten.
\newblock Greedy lattice animals: negative values and unconstrained maxima.
\newblock {\em Ann. Probab.}, 29(1):205--241, 2001.

\bibitem{DS}
R.~Durrett and R.~H. Schonmann.
\newblock Large deviations for the contact process and two-dimensional
  percolation.
\newblock {\em Probab. Theory Related Fields}, 77(4):583--603, 1988.

\bibitem{MR615434}
B.~Efron and C.~Stein.
\newblock The jackknife estimate of variance.
\newblock {\em Ann. Statist.}, 9(3):586--596, 1981.

\bibitem{GM-fpppc}
Olivier Garet and R{\'e}gine Marchand.
\newblock Asymptotic shape for the chemical distance and first-passage
  percolation on the infinite {B}ernoulli cluster.
\newblock {\em ESAIM Probab. Stat.}, 8:169--199 (electronic), 2004.

\bibitem{GM-large}
Olivier Garet and R{\'e}gine Marchand.
\newblock Large deviations for the chemical distance in supercritical
  {B}ernoulli percolation.
\newblock {\em Ann. Probab.}, 35(3):833--866, 2007.

\bibitem{Grimmett-Marstrand}
G.~R. Grimmett and J.~M. Marstrand.
\newblock The supercritical phase of percolation is well behaved.
\newblock {\em Proc. Roy. Soc. London Ser. A}, 430(1879):439--457, 1990.

\bibitem{MR2023652}
C.~Douglas Howard.
\newblock Models of first-passage percolation.
\newblock In {\em Probability on discrete structures}, volume 110 of {\em
  Encyclopaedia Math. Sci.}, pages 125--173. Springer, Berlin, 2004.

\bibitem{MR1452554}
C.~Douglas Howard and Charles~M. Newman.
\newblock Euclidean models of first-passage percolation.
\newblock {\em Probab. Theory Related Fields}, 108(2):153--170, 1997.

\bibitem{MR1849171}
C.~Douglas Howard and Charles~M. Newman.
\newblock Geodesics and spanning trees for {E}uclidean first-passage
  percolation.
\newblock {\em Ann. Probab.}, 29(2):577--623, 2001.

\bibitem{kesten-modere}
Harry Kesten.
\newblock On the speed of convergence in first-passage percolation.
\newblock {\em Ann. Appl. Probab.}, 3(2):296--338, 1993.

\bibitem{pimentel-preprint}
L.~Pimentel.
\newblock Asymptotics for first-passage times on {D}elaunay triangulations.
\newblock {\em preprint, available at
  \verb+http://arxiv.org/abs/math.PR/0510605+}, 2005.

\bibitem{MR840528}
J.~Michael Steele.
\newblock An {E}fron-{S}tein inequality for nonsymmetric statistics.
\newblock {\em Ann. Statist.}, 14(2):753--758, 1986.

\bibitem{MR1361756}
Michel Talagrand.
\newblock Concentration of measure and isoperimetric inequalities in product
  spaces.
\newblock {\em Inst. Hautes \'Etudes Sci. Publ. Math.}, (81):73--205, 1995.

\end{thebibliography}

%%%%%%%%%%%%%%%%%%%%%%%%%%%%%%%%%%%%%%%%%%%%%%%%%%%%%%%%%%%%

\end{document}